\documentclass{amsart}
\usepackage[mathscr]{eucal}
\usepackage{color}
\usepackage{xypic}
\usepackage{amsfonts}   
\usepackage{amsmath}
\usepackage{amsthm}
\usepackage{amssymb}
\usepackage{latexsym}

\let\nc\newcommand

\nc{\la}{\label}

\newtheorem{theorem}{Theorem}[section]
\newtheorem{definition}[theorem]{Definition}
\newtheorem{corollary}[theorem]{Corollary}
\newtheorem{lemma}[theorem]{Lemma}
\newtheorem{proposition}[theorem]{Proposition}
\newtheorem{example}[theorem]{Example}
\newtheorem{remark}[theorem]{Remark}

\def\k{\mathsf k}

\nc{\Hom}{{\rm{Hom}}}
\nc{\chara}{{\rm{char}}}
\nc{\Ext}{{\rm{Ext}}}
\nc{\HOM}{\underline{\rm{Hom}}}
\nc{\EXT}{\underline{\rm{Ext}}}
\nc{\TOR}{\underline{\rm{Tor}}}
\nc{\End}{{\rm{End}}}
\nc{\GL}{{\rm{GL}}}
\nc{\SL}{{\rm{SL}}}
\nc{\Rep}{{\rm{Rep}}}
\nc{\ad}{{\rm{ad}}}
\nc{\dlim}{\varinjlim}

\newcommand{\Frac}{{\rm{Frac}}}

\newcommand{\Spec}{{\rm{Spec}}}

\newcommand{\Autk}{{\rm{Aut}}}

\newcommand{\Der}{{\rm{Der}}}

\newcommand{\Ker}{{\rm{Ker}}}

\begin{document}
\title[Noncommutative Noether's problem]{Noncommutative Noether's problem\\ vs \\ classical Noether's problem}

\author{Vyacheslav Futorny}
\author{Jo\~ao Schwarz}
\address{Instituto de Matem\'atica e Estat\'istica, Universidade de S\~ao
Paulo,  S\~ao Paulo SP, Brasil} \email{futorny@ime.usp.br,}\email{jfschwarz.0791@gmail.com}

\begin{abstract}
We address the Noncommutative Noether's Problem on the invariants of Weyl fields for linear actions of finite groups. We prove that if the  variety $\mathbb A^n(\k)/G$ is rational 
then  the Noncommutative Noether's Problem  is  positively solved for $G$ and any field $\k$ of characteristic zero. In particular, this gives positive solution for  all  pseudo-reflections groups, for the alternating groups ($n=3, 4, 5$) and for any finite group when $n=3$  and $\k$ is algebraically closed. Alternative proofs are given for the complex field and for  all  pseudo-reflections groups. In the later case an effective algorithm of finding the Weyl generators is described. 
We also study birational equivalence for the rings of invariant differential operators on
 complex affine irreducible varieties.

\end{abstract}

\maketitle

\section{Introduction}
We will assume that all algebras are considered over the field $\k$ of characteristic zero. 

Let $G$ be a finite group acting linearly on the ring of polynomials $\k [x_1, \ldots, x_n]$.  The ring of invariants $\k [x_1, \ldots, x_n]^G$ is polynomial if and only if $G$ is a finite group generated by pseudo-reflections by the Chevalley-Shephard-Todd theorem. The pseudo-reflection groups were classified  by Shephard and Todd. When $\k= \mathbb{Q}, \mathbb{R}, \mathbb{C}$ they are the Weyl groups, the finite Coxeter groups and the unitary reflection groups, respectively. 
Geometrically, the Chevalley-Shephard-Todd theorem solves the problem of description of all finite groups $G$ acting linearly on the affine space for which the quotient variety $\mathbb{A}^n(\k)/G$ is biregular to $\mathbb{A}^n(\k)$.

Extend the linear action of  $G$ on $\k [x_1, \ldots, x_n]$ to the action on the field of rational functions $K_n=\k(x_1, \ldots, x_n)$.  The birational version of the problem above leads to the classical Noether's Problem (CNP for short) related to the $14$-th Hilbert's problem, posing a question whether   $K_n^G$ is a purely transcendental extension of $\k$ or, equivalently,  whether $\mathbb{A}^n(\k)/G$ is birational to $\mathbb{A}^m(\k)$ for some $m$. In the latter case we say that the quotient variety $\mathbb{A}^n(\k)/G$ is \emph{rational}.

Well known  cases with positive solution for the CNP include  $n=1$ (Luroth), $n=2$ (Miyata) or $n=3$ and $\k$ is algebraically closed (Burnside)  and   any $n$ when the representation of $G$ is isomorphic to a direct sum of one
dimensional representations (Fischer). It also holds for alternating groups  $G=A_n$ for $n=3, 4, 5$ (Maeda).   Detailed references including the counter-examples  of Swan, Voskresenskii and Saltman  can be found in \cite{Dumas} and \cite{Jensen}.

Passing to a noncommutative case consider the $\k$-algebra of differential operators on the polynomial ring $\k [x_1, \ldots, x_n]$, which is the $n$-th Weyl algebra $A_n(\k)$, and extend   the action of  $G$ to a linear action
on $A_n(\k)$.  The algebra $A_n(\k)$ is a simple Ore domain with admits the skew field of fractions which we will denote by $F_n(\k)$.  The action of  $G$ extends naturally to $F_n(\k)$.

An analog of the Noether's Problem for the Weyl algebra $A_n$ was first considered by Alev and Dumas in \cite{AD1}, where they asked whether $F_n(\k)^G$ is isomorphic to $F_m(L)$ for some $m$ and some purely transcendental extension 
$L$ of $\k$. In fact, if such isomorphism holds, then $m=n$ and $L=\k$.  This is the case, for instance,  for  $n=1$ and $n=2$ and an arbitrary finite group $G$ \cite{AD1}, for any $n$ and any $G$ whose natural representation decomposes into a direct sum of one dimensional representations \cite{AD1}. In particular,  it holds for all $n \ge 1$ if $G$ is abelian and $k$ is
algebraically closed. It was shown in \cite{Eshmatov} that it also holds 
 for any $n$ and any complex reflection group.

Due to the importance  of Weyl algebras and to the fact that they are the simplest noncommutative deformations of polynomial algebras,  we call the analog of the Noether's Problem for $A_n$ the \emph{Noncommutative Noether's Problem} (NNP for short).

The cases when the NNP has a positive solution are of special interest in view  of the rigidity of the Weyl algebras  proven by Alev and Polo \cite{AP}: $A_n(\k)^{G}$ is not isomorphic to $A_n(\k)$  when $\k$ is algebraically closed,  
for any non trivial linear action of $G$.     This was recently generalized by  Tikaradze \cite{Tikaradze}, who showed that $A_n(\mathbb{C})$ is not isomorphic to any invariant subring $D^G$ for any domain $D$ and any finite group $G$ unless $D\simeq A_n(\mathbb{C})$ and $G$ is trivial.  Hence, positive solutions of the NNP give examples for the question posed by Kirkman, Kuzmanovich and Zhang \cite{KKZ}, asking for which rigid algebras the skew field of fractions and its  skew subfield of invariants are isomorphic.  In this case the algebra and its subalgebra of invariants are called birationally equivalent.
The Noncommutative  Noether's Problem is also connected to the Gelfand-Kirillov Conjecture on the birational equivalence between the universal enveloping algebras and Weyl algebras. It can be used to reprove the Gelfand-Kirillov Conjecture for $gl_n$ and $sl_n$ \cite{Futorny} and show it for all finite $W$-algebras of type $A$ \cite{FMO}.  For quantum versions of the Noether's Problem and  the Gelfand-Kirillov Conjecture we refer to \cite{FH} and \cite{Hartwig}.  

We will say that the NNP holds for a group $G$ if it has the positive solution for $G$. We will also say that the CNP \emph{implies} the NNP if whenever the CNP holds for a group $G$, the NNP also holds for the extended action of $G$.  
The first goal of our paper is to prove the following statement which was first conjectured in less general form  in \cite{Schwarz}

\begin{theorem}\label{main}
For any field $\k$ of zero characteristic and any linear action of a finite group $G$ if the quotient  variety $\mathbb A^n(\k)/G$ is rational then the NNC holds for $G$, that is
the CNP implies the NNP.
\end{theorem}

The was known to be true for $n=1$ and $n=2$, for any $n$ when the natural representation of $G$ is isomorphic to a direct sum of one dimensional representations \cite{AD1} and for any complex reflection group \cite{Eshmatov}.

As an immediate consequence of this result we obtain that the Noncommutative Noether's Problem  holds for \emph{all pseudo-reflection groups},
 \emph{alternating groups  $\mathcal A_3$, $\mathcal A_4$, $\mathcal A_5$} and for \emph{any finite group}  when $n=3$ and $\k$ is algebraically closed.

We consider separately the case of the complex field. Let    $X$ be  a complex affine irreducible variety, $G$  a finite group of automorphisms on $X$ and let $D(X)$ be   the algebra of differential operators  on $X$.
Our second main result is the following

\begin{theorem}\label{main1}
Let $X$ is an affine irreducible variety.  If $X/G$ is birationally equivalent to an irreducible  affine $Y$ then $$\Frac (D(X)^G) \cong \Frac  (D(Y)).$$
\end{theorem}

In a particular case of rational varieties we have

\begin{theorem}\label{main2} 
Let $X$ be a complex affine  irreducible variety and $G$  a  finite group of automorphisms on $X$.  Suppose $X/G$ is rational.
Then $$(\Frac (D(X)))^G\simeq F_n(\mathbb C).$$
\end{theorem}

Using this theorem above we obtain a different proof of Theorem \ref{main} for the field of complex numbers.

Next we give an alternative proof of Theorem \ref{main} for all  pseudo-reflections groups and an arbitrary field of characteristic zero (Theorem \ref{NNP-pseudo-reflections}) extending the results of \cite{Eshmatov} on complex irreducible reflections groups. This approach has an advantage since it allows to find the Weyl generators  by a fairly simple algorithmic procedure. 

Finally we apply the results to birational equivalence of the cross products (Theorem \ref{thm-cross}).

\

\noindent{\bf Acknowledgements.}  V.F. is
supported in part by  CNPq grant (304467/2017-0) and by 
Fapesp grant (2014/09310-5).  J.S. is supported in part by Fapesp grant (2014/25612-1). The second author is very grateful to Jacques Alev for attention and fruitful discussion.

\section{Preliminaries}

\subsection{Weyl algebras}

We will denote by $A_n(\k)$  the $n$-th Weyl algebra over the field $\k$, which is the unital associative algebra generated   over  $\k$  by the elements $x_1, \ldots, x_n$, $\partial_1, \ldots, \partial_n$ subject to the relations 
$\partial_i  x_j-x_j \partial _i=\delta_{ij}$,   $x_ix_j=x_jx_i$, $\partial_i \partial_j=\partial_j \partial_i$ for $1\leq i,j\leq n$. 

For each $i=1, \ldots, n$ denote $t_i=\partial_i x_i$ and consider $\sigma_i\in \Autk \  \k[t_1, \ldots, t_n]$ such that
$\sigma_i(t_j)=t_j-\delta_{ij}$ for all $j=1, \ldots, n$. Let $\mathbb Z^n$ be the free abelian group generated by $\sigma_1, \ldots, \sigma_n$. Then we  have a natural embedding $A_n(\k)\rightarrow  \k[t_1, \ldots, t_n]*\mathbb Z^n$ where 
$x_i\mapsto \sigma_i$, $\partial_i \mapsto t_i\sigma_i^{-1}$, $i=1, \ldots, n$.. Moreover, $$A_n[S^{-1}]\simeq \k(t_1, \ldots, t_n)*\mathbb Z^n,$$ where $S=\k[t_1, \ldots, t_n]\setminus \{0\}$ \cite{Futorny}.

 

Let $X$ be an affine variety over $\k$ with the coordinate ring  $\mathcal O(X)$.  The ring of differential operators  $D(X)$   on $X$ (and on $\mathcal O(X)$) is defined as
 $D(X)=\cup_{n=0}^{\infty} D(X)_n$, where $D(X)_0=\mathcal O(X)$ and
$$ D(X)_n \, = \, \{ \, d \in \End_{\k} (\mathcal O(X)) \, :\, d\, b - b\, d \in D(X)_{n-1}\, \mbox{ for all }\, b \in \mathcal O(X) \}.$$ 
In particular, $\mathcal O(X)\subset D(X)$.
For smooth $X$, 
 $D(X)$ is 
the subalgebra of $End_{\k} (\mathcal O(X))$ generated  by the $\k$-linear derivations of $\mathcal O(X)$ and the scalar multiplications $l_f$, $f\in \mathcal O(X)$, that sends $g \rightarrow fg$ for any $g \in \mathcal O(X)$.  

The fact that at $D(X)$ is an Ore domain for irreducible and smooth $X$ is well known (e.g., \cite{McConnell}). But in fact this holds without the latter assumption, as the following proposition shows:

\begin{proposition}
Let $X$ be an affine irreducible variety over $\k$. Then $D(X)$ is an Ore domain.
\begin{proof}
Call $K:= \, Frac \, O(X)$. We can realize $D(X)$ as a subset of $D(K)$ in the following way (\cite{McConnell}, 15.5.5(iii)):

\[ D(X) = \{ d \in D(K)|d(O(X)) \subset O(X) \}. \]
Now, since $K$ is a regular ring, $D(K)$ is a non-commutative domain with finite Gelfand-Kirillov dimension; since $D(X)$ is a subring of it, the same properties hold for it. Hence, $D(X)$ does not contain a subalgebra isomorphic to the free associative algebra in two variables. It follows, then, by a result of Jategaonkar (\cite{KL}, Prop. 4.13), that $D(X)$ is an Ore domain.
\end{proof}
\end{proposition}

The Weyl algebra $A_n(\k)$ is isomorphic to the the ring of differential operators on the affine space $X=\mathbb A^{n}$ (or, equivalently, on the polynomial algebra $\mathcal O(X)$ in $n$ variables).

\subsection{Subalgebras of invariants}

For a ring $R$ and a group $G$ of automorphisms of $R$ denote by $R^G$ the subring of $R$
consisting of elements fixed by every element of
$G$, so-called the Galois subring corresponding to $G$.  It was shown by Faith \cite{Fa} that  if  $R$ is a right Ore domain then any 
subring $R^G$  corresponding to a finite group $G$ is also
right Ore domain. Assuming $R$ to be both right and left Ore domain, it admits the skew field of fractions which we denote by $\mathcal F (R)$.  Then the action of $G$ on $R$ extends uniquely to an action on $\mathcal F (R)$ and $(\mathcal F (R))^G\simeq \mathcal F (R^G)$. Hence, $\mathcal F: R\mapsto \mathcal F (R)$ defines a functor from the category of Ore domains  with injective homomorphisms to the category of skew fields. If $\phi: R_1\rightarrow R_2$ is a morphism of domains such that $\mathcal F (\phi)$ is an isomorphism of skew fields then we say that $R_1$ and $R_2$ are birationally 
equivalent. 

If $\phi: R_1 \rightarrow R_2$ is a $G$-equivariant morphism then it induces a morphism $\phi^G: R_1^G \rightarrow R_2^G$. If, in addition, $\mathcal F (\phi)$ is an isomorphism then 
$R_1^G$ and $R_2^G$ are birationally 
equivalent. 

Since $A_n(\k)$ is a Noetherian Ore domain, it admits the skew field of fractions $F_n(\k):=\mathcal F(A_n(\k))$, called \emph{the Weyl field}.

\section{Algebras of differential operators}

Let $X$ be an affine variety over $\k$ with an action of a finite group $G$.  Then this action 
 can be extended to the ring of differential operators $D(X)$ on $X$. Indeed, the group  $G$ acts by algebra automorphisms  on $\mathcal{O}(X)$. The extension of this action to the ring of differential operators $D(X)$ is done as follows: if $d\in D(X)$ then  $ (g \ast d) \cdot f = (g \circ d \circ g^{-1}) \cdot f$    for any $f \in \mathcal{O}(X)$. The elements of $D(X)$ invariant under the action of $G$ 
 are called \emph{$G$-invariant differential operators}.


The following result was established by Cannings and Holland \cite{CH}:

\begin{theorem} \label{C-H} Let $\k$ be the field of complex numbers.
Let $X$ be an affine irreducible algebraic variety over $\mathbb C$ with an action of a finite group $G$  on it. 
\begin{itemize}
\item[1)]
There exits a maximal open dense  $G$-invariant  subset $V\subset X$, on which the induced action of $G$ is free. 
Let  $\pi : X \rightarrow X/G$ be the canonical projection  and $V'=\pi(V)$.  Then $V'$  is open dense in $X/G$ and, since $V$ is a complete pre-image, $V = \pi^{-1}(V)$,  the map $\pi$ restricts to the quotient map:  $$\pi |_{V}: V\rightarrow V'.  $$ 
\item[2)]
Let $V$ be an open subset of $X$ on which the action of $G$ is free. If $(\dagger )$ $X$ satisfies the Serre condition $S_2$ (in particular, if it is normal) and 
$$ \, codim_X(X-V) \geq 2$$ then $$D(X/G) \cong D(X)^G.$$ The same isomorphism holds if $(\ddagger)$ $G$ acts freely.
\end{itemize}
\end{theorem}

Denote by $\mathcal{O}(X)$ the coordinate ring of $X$. Recall that the variety $X$ satisfies the  $S_2$ condition of  Serre if  $depht \, \mathcal{O}(X)_p \geq inf\{2, height(p) \}$ for all prime ideals $p$ of $\mathcal{O}(X)$ (\cite{Matsumura}).

\begin{lemma} \label{main-lemma}
Let $X$ be an irreducible affine variety. Suppose that a finite group $G$ is acting by automorphisms on $X$ and either $\dagger$ or $\ddagger$ hold. If $X/G$ is birationally equivalent to a affine irreducible variety $Y$ then $\Frac (D(X)^G) \cong \Frac (D(Y))$.
\end{lemma}

\begin{proof}
Let $S$ be the set of regular elements in $\mathcal{O}(X)^G$. Since $X/G$ is birational to $Y$  we have $$\Frac (\mathcal{O}(X)^G) =\Frac (\mathcal{O}(X)^G_S) \cong \Frac (\mathcal{O}(Y)).$$ Since 
$\Frac (D(X)^G) \cong \Frac (D(X/G)) $ by Theorem \ref{C-H}, then
 applying \cite{Muhasky}, Proposition 1.8, we have
$$\Frac (D(X)^G)  \cong \Frac (D(X/G)_S) \cong \Frac (D(\mathcal{O}(X)^G_S)) \cong \Frac D(\Frac (\mathcal{O}(Y))), $$ and hence
$\Frac (D(X)^G)  \cong \Frac (D(Y)).$ 
Note that unlike in a similar statement  in \cite{McConnell} we do not assumed the variety $X$ to be smooth.
\end{proof}

\section{CNP implies NNP}

In this section we prove our main result:  the CNP implies the NNP for an arbitrary field of characteristic zero and arbitrary  linear action of a finite group. 

Consider an arbitrary finite group $G$ acting linearly on an $n$-dimensional $\k$-vector space $V$ and  its naturally extended action on $\mathcal{O}(V^*) = \k[x_1,\ldots,x_n]$.
This action extends to  the Weyl algebra $A_n(\k)\simeq D(\mathcal{O}(V^*))$.  Recall that a positive solution for the Classical Noether's Problem for this action means that $\Frac (\mathcal{O}(V^*)^G) \cong \k(x_1,\ldots, x_n)$.

Denote $\mathcal{B}$ the subalgebra of $A_n(\k)^G$ generated by $\mathcal{O}(V^*)^G =\k[x_1, \ldots, x_n]^G$ and $\mathcal{O}(V)^G = \k [\partial_1, \ldots, \partial_n]^G$.
 Then $\mathcal{B}=A_n(\k)^G$ by  \cite{Levasseur}, Theorem 5. We will closely follow the argument in the proof of this fact. 

Set $S= \mathcal{O}(V^*)^G \setminus \{0 \}$. Since $S$ is ad-nilpotent 
on $A_n(\k)$, and hence on $\mathcal{B}=A_n(\k)^G$, it is an Ore set in both algebras (\cite{KL}, Theorem 4.9).

Denote $F:= \Frac (\mathcal{O}(V^*)^G)$.

Recall the following lemma from \cite{Levasseur}:

\begin{lemma}\label{lemma-curious}[\cite{Levasseur}, Lemma 8]
Let $L$ be a finite field extension of $k$, with $tdg_k \, L = l$, and consider the ring of differential operators on $L$, $D(L)$. Let $A$ be a subalgebra of $D(L)$ containing $L$, with a filtration induced from that of $D(L)$ (by order of differential operator). If the associated graded algebra of $A$ contains as a subalgebra a finitely generated graded $L$-algebra $B$ such that $Kdim \, B = l$, then $A=D(L)$.
\end{lemma}

\begin{lemma}\label{lemma-MAIN}
$\Frac (A_n(\k)^G) \cong \Frac (D(F))$.
\end{lemma}

\begin{proof}
We have $\mathcal{B} \subset A_n(\k)^G$ by definition. On the other hand,  $A_n(\k)^G \subset D(\mathcal{O}(V^*)^G)$ by restriction of domain. We have $$D(\mathcal{O}(V^*)^G)_S = D(\mathcal{O}(V^*)^G_S)= D(F)$$ by \cite{Muhasky}, Proposition 1.8.
After localization by $S$ we obtain:

\[  \mathcal{B}_S \subset A_n(\k)^G_S \subset D(\mathcal{O}(V^*)^G)_S = D(F). \]

Consider the filtration on $\mathcal{B}_S$ induced from $D(F)$. Since $\mathcal{O}(V^*)^W \subset \mathcal{B}$, we have that $gr \, \mathcal{B}_S$ contains $F \otimes \mathcal{O}(V^*)^G $ as a graded $F$-subalgebra. Since $\mathcal{O}(V^*)$ is finite over $\mathcal{O}(V^*)^G$ then it has the Krull dimension $n$ (\cite{McConnell}).  Applying Lemma \ref{lemma-curious} we have  $\mathcal{B}_S = D(F)$.
We conclude that $\Frac (D(F)) \subset \Frac (A_n(\k)^G) \subset \Frac (D(F))$ which implies the desired equality.
\end{proof}

\begin{remark}
Note that in fact in the above proof we do not need the equality $\mathcal{B}=A_n(\k)^G$ but only the 
embedding $\mathcal{B} \subset A_n(\k)^G$ by \cite{Levasseur}, Lemma 9.
\end{remark}

As a consequence of Lemma \ref{lemma-MAIN} we immediately obtain our main result.

\begin{theorem}\label{thm-main}
The CNP implies the NNP for any linear representation of a finite group over any field of characteristic zero.
\end{theorem}

\begin{proof}
Under the condition of the theorem we have 
that $F \cong \k(x_1,\ldots,x_n)$. Then $D(F)$ is isomorphic to the localization of $A_n$ by $x_1, \ldots, x_n$ and the statement follows.
\end{proof}

We immediately  have the following application of the main result

\begin{corollary}
The Noncommutative Noether's Problem  holds in the following cases or any field of characteristic zero:
\begin{itemize}
\item for all linear representations of all pseudo-reflection groups;
\item for  alternating groups $\mathcal{A}_n$ with  usual permutation action  for  $n=3,4,5$; 
\item for any group  when $n=3$  and $\k$ is algebraically closed.
\end{itemize}
\end{corollary}

We note that Theorem \ref{thm-main} allows to recover all results for finite groups from \cite{AD1}  and all results from \cite{Eshmatov}. 
Also it allows  us to give a shorter proof of the following fact shown in \cite{AD1}:

\begin{theorem}[\cite{AD1}] If $F_n(\k)^G$ is isomorphic to $F_m(L)$ for some $m$ and some purely transcendental extension $L$ of $\k$ of transcendence degree $t$, then $m=n$ and $t=0$.
\end{theorem}

\begin{proof}
 We have that $F_m(L)\simeq \Frac (D(F))$. Now  use  \cite{Dumas}, Lemma 3.2.2.
 The  center of $F_m(L)$ has the transcendence degree $t$ over $\k$. By the primitive element theorem, $F=\k (y_1,\ldots, y_n) (f)$, for certain algebraically independent $y_1,\ldots,y_n$, and by \cite{McConnell} 15.2.4, the second skew-field has center of transcendence degree 0. So $t=0$ and $m=n$. 
\end{proof}

\section
{Complex $\k $ case}

In this section we prove Theorem \ref{main1}  and Theorem \ref{main2}.

Suppose $X$ is an affine irreducible variety.  Fix  a finite group $G$ acting on $X$ and satisfying the hypnotises of the theorem.
By Theorem \ref{C-H}, 1), there exists an open dense subset $V$ on which the action of $G$ is free and such that the quotient map $\pi: X \rightarrow X/G$ restricts to the quotient $\pi:V \rightarrow V'$, where $V'$ is open dense in $X/G$. By the Hilbert-Noether theorem, the map $\pi$ is finite, hence affine (\cite{Hartshorne}, Exercise 5.17). Let $W'$ be a principal open subset of  $V'$. Since $\pi$ is  affine,  then $W=\pi^{-1}(W')$ is affine.     Since $W$ is a union of orbits,  $G$ restricts to a free action on it. We now  have a quotient map $\pi: W \rightarrow W'$ with $W$ affine and smooth (hence normal). Also, since $W' \subset X/G$, then $W'$ is birational to 
$Y$. 
Applying Lemma \ref{main-lemma} we obtain

\begin{lemma} \label{lemma-intermediate}
$\Frac (D(W)^G) \cong \Frac (D(Y))$.
\end{lemma}

For $f\in \mathcal{O}(X)$ denote $Spec \, \mathcal{O}(X)_f$ the  principal open subset. These sets constitute a basis of the Zariski topology, and hence there exists a principal open subset  $Spec \, \mathcal{O}(X)_h \subset W$ for $h\in \Lambda$. 
Set $f=\prod_{g \in G} g.h$. Then $f$ is $G$-invariant. Thus we have

\begin{lemma}\label{lem-f}
There exists a principal open set $Spec \, \mathcal{O}(X)_f\subset W$ with $G$-invariant $f$.
\end{lemma}

 Now we generalize the argument given in the proof for  unitary reflection groups in \cite{Eshmatov}.
By Lemma \ref{lem-f}
there exists a principal open set $Spec \, \mathcal{O}(X)_f\subset W$ with $G$-invariant $f$. Then
we have the following inclusions of varieties: $\Spec \, \mathcal{O}(X)_f \subset W \subset X$.  Let $D(.)$ be the sheaf of differential operators functor 
which associates the ring of differential operators to a given variety. Functor $D(.)$ is contravariant and 
 we have chain of inclusions $$ D(\mathcal{O}(X)) \subset D(W) \subset D(\mathcal{O}(X)_f)=D(\mathcal{O}(X))_f$$
 (cf. \cite{Liu}, Proposition 2.4.18). Taking the field of fractions, and then the $G$-invariants, we have the following chain:

$$\Frac  D(\mathcal{O}(X))^G \subset \Frac (D(W))^G \subset \Frac (D(\mathcal{O}(X))_f)^G$$ 
$$=  \Frac (D(\mathcal{O}(X))^G_f) = \Frac( D(\mathcal{O}(X))^G).$$

Then applying Lemma \ref{lemma-intermediate} we have
$$\Frac (D(X)^G) \simeq \Frac (D(W)^G) \cong \Frac (D(Y)),$$   which implies the statement of Theorem \ref{main1}. Theorem \ref{main2} follows immediately, since in this case we have $Y =\mathbb{A}^n(\k)$, $n=dim \, X$.

\section{NNP for pseudo-reflection groups}

 By the Chevalley-Shephard-Todd theorem the CNP holds 
for all pseudo-reflection groups over any field of zero characteristic.
In this section we give an alternative proof that the Noncommutative Noether's Problem has a positive solution for all pseudo-reflection groups over any field of zero characteristic, which is of independent interest.
For complex reflection groups this was shown in \cite{Eshmatov}, Theorem 2.

As before $\Lambda$ denote the polynomial algebra over $\k$ with $n$ variables. Let $W$ be an arbitrary pseudo-reflection group acting by linear automorphisms on $\Lambda$.
Recall the following statement [\cite{Eshmatov}, Proposition 1] which does not depend on the field $\k$:

\begin{proposition} \label{elem} Let $\Delta$ be a $W$-invariant element of $\Lambda$,  $S$  a multiplicatively closed set in  $\Lambda$. Then
\begin{enumerate}
\item
$(\Lambda_\Delta)^W = (\Lambda^W)_\Delta$;
\item
$D(\Lambda_S) = D(\Lambda)_S$;
\item
$(D(\Lambda)_\Delta)^W \cong (D(\Lambda)^W)_\Delta$.
\end{enumerate}
\end{proposition}

Consider  a $W$-invariant element $\Delta\in \Lambda$, localization $\Lambda_\Delta$ with the induced action of $W$ and the $W$-invariants $\Lambda_\Delta^W$ in $\Lambda_\Delta$. 
We have an embedding   $\Lambda_\Delta^W \rightarrow \Lambda_\Delta$.  By the restriction of domain we have an induced map $$ \phi_{\Delta}: D(\Lambda_\Delta)^W \rightarrow D(\Lambda_\Delta^W).$$

\begin{proposition} \label{dagger}
Let $\Delta$ be a $W$-invariant element in $\Lambda$. Then the map $\phi_{\Delta}$ is injective.
\end{proposition}

\begin{proof}
Note that  $D(\Lambda_\Delta)$ is a simple ring and $W$ acts by outer automorphisms.   Then $D(\Lambda_\Delta)^W$ is a simple ring, by \cite{Montgomery}, Corollary 2.6. 
Since $\phi_{\Delta}$ is not trivial, it is injective. 
\end{proof}

Our goal now is to find an adequate $\Delta$ such that the $\phi_{\Delta}$ is surjective. The case  $W= S_n$ was considered in \cite{FS}.

\subsection{Proof of the NNP for irreducible pseudo-reflection groups}
We proceed by considering first irreducible pseudo-reflection groups.  
Recall that
a pseudo-reflection group $W$ is called \emph{irreducible} if its natural representation is irreducible.

We will make use of the following notion of the \emph{field of definition} of $G$ - the smallest subfield where a representation of the group $G$ is defined.
More precisely,

\begin{definition}
Let $\rho: G \rightarrow GL_n(\k)$ be a linear representation of a finite group $G$. Let $\k' \subset \k$ be a subfield. Suppose there exists a homomorphism $\rho': G \rightarrow GL_n(\k')$ such that $\rho$ can be obtained from $\rho'$ by the extension of scalars.
We say that $\rho$ has $\k'$ as the field of definition if $\k'$ is the smallest subfield with this property. 
\end{definition}

Given a linear representation $\rho: G \rightarrow GL_n(\k)$ denote by $\chi_{\rho}$ the corresponding character function.  Let $\mathbb{Q}(\chi)$ be the field
 extension of $\mathbb{Q}$ by $Im \, \chi$.    

By \cite[Appendix B]{Kane},  we have

\begin{proposition} \label{irred} Let $W$ be an irreducible pseudo-reflection group and
 $\rho: W \rightarrow GL_n(\k)$ a representation of $W$. Then $\rho$ has $\mathbb{Q}(\chi)$ as the field of definition.
\end{proposition}

We shall also need the following fact from the invariant theory of pseudo-reflection groups.
Let $M$ be the $n \times n$ matrix whose $ij$'s entry is $\partial_{x_j}e_i$, where $\Lambda^W= \k[x_1,\ldots,x_n]^W = \k[e_1,\ldots, e_n]$. Let  $J'$ be the determinant of $M$.

Let $\mathcal{S}$ be the set of all pseudo-reflections in $W$. Each $s \in \mathcal{S}$ fixes a hyperplane $H_s$. Let $L_s$ be a linear form whose kernel is $H_s$ for each $s \in \mathcal{S}$. 
Set $J = \prod_{s \in \mathcal{S}} L_s$. It has the following properties:

\begin{proposition}\label{quasi-invariant}[\cite{Kane}, 20-2, Proposition A and B, 21-1, Proposition A and B]
$J \neq 0$ and  $w.J=det(w)J$ for every $w \in W$. Moreover, $J$ is a multiple of  $J'$.
\end{proposition}

As in the case of complex reflection groups set $\Delta=J^{|W|}$ (\cite{Eshmatov}, Section 3).

Let $E_i$, $i=1, \ldots, n$ be the column vector, where we have $1$ in the $i$th position and $0$ in all others. Let \[ F_i =  \left( \begin{array}{c} f_{i1} \\ \vdots \\ f_{in} \end{array}\right) \] be a solution of the linear system  $M F_i = E_i$. By the Kramer's rule, $f_{ij} \in \Lambda_J$, $1 \leq i,j \leq n$, where $\Lambda_J$ is the localization of $\Lambda$ by $J$.

For each $i=1, \ldots, n$ set $d_i = \sum_{k=1}^n f_{ik}\partial_k$. Then  $d_i \in D(\Lambda_\Delta) = D(\Lambda)_\Delta$ and we have $d_i(e_j) = \delta_{ij}$, $i,j=1, \ldots, n$.

We will show  that all differential operators $d_i$, $i=1, \ldots, n$ are  $W$-invariant. By Theorem \ref{irred} we can assume that $e_i$'s, and hence $d_i$'s, have coefficients in $\mathbb{Q}(\chi)$.
Observe the following: let $\k' \subset \k$ be a subfield fixed by $W$ and $d$ a differential operator with coefficients in $\k'$, then the question of $W$-invariance of $d$ is the same, weather we consider the base field $\k$ or $\k'$. 
As $\mathbb{Q}(\chi)$ is fixed by $W$, to show that the $d_i$'s are invariant differential operators on $\Lambda_\Delta^W$, we can replace $\k$ by $\mathbb{Q}(\chi)$. Now our field of definition is a subfield of $\mathbb{C}$.
Repeating the above argument we can assume that $\k = \mathbb{C}$.

Recall the following result of Knop:

\begin{theorem}\label{thm-knop}[\cite{Knop}, Theorem 3.1]
Let $X$ be a complex affine irreducible  normal variety. Then $D(X)^W=\{d \in D(X) | d(\mathcal{O}(X)^W) \subseteq \mathcal{O}(X)^W \}$.
\end{theorem}

Denote by $\Delta'$ the result of replacing every $x_i$ in $\Delta$ by $e_i$ for all $i=1, \ldots, n$.
By the Chevalley-Shephard-Todd theorem, $\Lambda_\Delta^{W} \simeq \k[e_1,\ldots,e_n]_{\Delta'}$. Taking into account the action of operators $d_i$'s  and Theorem \ref{thm-knop} we obtain the desired invariance of  $d_i$'s under the action of $W$. Indeed,   for each  $i=1, \ldots, n$ the operator $ d_i$ sends every  $e_j$ to an element of $\k[e_1,\ldots,e_n]_{\Delta'}$ and  the same holds for $\Delta'$.   Since these elements generate $\k[e_1,\ldots,e_n]_{\Delta'}$ the statement follows.

\begin{proposition} \label{rel-Weyl}
 The map  $ \phi_{\Delta}: D(\Lambda_\Delta)^W \rightarrow D(\Lambda_\Delta^W)$ is surjective.
 \end{proposition}
 
 \begin{proof}
It is sufficient to show that the images of $d_i$, $e_i$, $i=1, \ldots, n$ under $ \phi_{\Delta}$ are  the Weyl generators of $D(\Lambda_\Delta^{W})$. Let $A:=\Lambda_\Delta^{W} $.
 The $A$-module of K\"{a}lher differentials $\Omega_{\k}(A)$ is freely generated over $A$ with basis $d_{e_1}, \ldots, d_{e_n}$. Then, by \cite{McConnell}, 15.1.12, the  $A$-module of derivations $\Der_{\k}(A)$  is freely generated by the unique extensions of $\partial_{e_i}$, $i=1, \ldots, n$ from $\k[e_1, \ldots, e_n]$ to $A$. Clearly,
  $\phi_{\Delta}(d_i)=\partial_{e_i}$, $i=1, \ldots, n$.
\end{proof}

Combining Proposition \ref{dagger} and Proposition \ref{rel-Weyl} we conclude
$$ D(\Lambda_\Delta)^W \simeq D(\Lambda_\Delta^W).$$

Applying Proposition \ref{elem} we finally have

\begin{corollary}\label{NNP-irred}
Let $\k$ be an arbitrary field of zero characteristic and $W$ an irreducible pseudo-reflection group. Then the
NNP holds for $W$.
\end{corollary}

\subsection{NNP for general pseudo-reflection groups}
In this subsection we consider general pseudo-reflection groups.

Let  $V$ be a finite dimensional vector space. If  $g$ is
 a linear automorphism  of $V$ then we set $Fix \, g = \{ v \in V | gv= v \} = \Ker \, (Id - g)$, and $[V,g]= Im(Id - g)$.

If $g$ is a pseudo-reflection, $g\neq id$ then $Fix \, g$ is a hyperplane and $[V, g]$ is one dimensional. If $a \in V$ generates $[V, g]$ then for every $v \in V$ there exists $\psi (v) \in \k$ such that $v - gv = \psi (v) a$.  Then $\psi$ is a linear functional on $V$ and $\Ker \, \psi = Fix \, g$.

The following is standard

\begin{lemma} \label{boring-4.1}
If $g,h \in GL(V)$ then $Fix(ghg^{-1}) = g Fix \, h$. If $h$ is a pseudo-reflection with the fixed hyperplane $H$, then $g h g^{-1}$ is also a pseudo-reflection with the fixed hyperplane $gH$.
\end{lemma}

In the following we collect basic properties of pseudo-reflections.

\begin{lemma}\label{lem-boring}    
\begin{itemize}
\item[(1)]
\label{boring-4.2}
Let $g$ be a pseudo-reflection of order $m>1$, $H = Fix \, g$,  $L_H$ any linear functional such that $H = \Ker \, L_H$. Let $a$ be a generator of $[V, g]$. Then there exists an $m$-th primitive root of unity $\mu$ such that $gv = v - (1 - \mu)\frac{L_H(v)}{L_H(a)}a$, for all $v \in V$.
\item[(2)] \label{boring-4.3}
Let $r, s \neq id$ be pseudo-reflections, $H = Fix \, r$, $J = Fix \, s$, $x$ a generator of $[V, r]$, and $y$ a generator of $[V, s]$. If $x \in J$ and $y \in H$ then $rs=sr$.
\item[(3)] \label{boring-4.4}
A subspace  $V' \subset V$ is invariant by a pseudo-reflection $g \neq id$ if and only if $V' \subseteq Fix \, g$ or $[V, g] \subseteq V'$.
\end{itemize}
\end{lemma}

\begin{proof} Given a pseudo-reflection $g$ 
consider a linear functional $\psi$ such that $gv = v - \psi(v) a $ and $H = \Ker \, \psi$, as above. Hence $ga= \mu a$ for a primitive $m$-th root of unity $\mu$, and hence $\psi(a) = 1 - \mu$. We have $\psi = \lambda L_H $, where  $0 \neq \lambda \in \k$. This gives  $\lambda = (1 - \mu)/L_H(a)$ and implies statement (1).
Applying (1),c we have $\mu, \nu \in \k$ such that $\forall \, v \in V$
\[ rs(v) = v - (1 - \mu) \frac{L_H(v)}{L_H(a)}a - (1 - \nu) \frac{L_J(v)}{L_J(b)} b + (1 - \mu)(1 - \nu) \frac{L_H(b) L_J(v)}{(L_H(a) L_J(b))} a .\]
If $y \in H$ then $L_H(y)=0$ and the last term is $0$. Analogously, the last term in the expression of $sr(v)$ is 0 and other terms in both expressions
are equal. Therefore $rs=sr$.

Finally, if $V' \subseteq Fix \, g$ or $[V, g] \subseteq V'$, then  $V'$ is invariant by  the statement (1). Conversely,  if $V'$ is $g$-invariant and is not contained in $Fix \, g$, then $[V', g] \neq 0$, and hence $[V, g]=[V', g] \subseteq V'$.   
\end{proof}

The following is probably well known but we include the proof for the sake of completeness.

\begin{proposition} \label{coxeter-like}
Let $W$ be a finite group of pseudo-reflections on $V$. Consider a decomposition $V_1 \oplus \ldots \oplus V_m$  of the $\k W$-module $V$ into irreducible submodules and set $W_i$ to be the restriction of $W$ to $V_i$, $i=1, \ldots, m$. Then $W_i$ is either a pseudo-reflection group or trivial, and $W \simeq W_1 \times \ldots \times W_m$.
\end{proposition}

\begin{proof}
By Lemma \ref{lem-boring}, (3), if $g$ is a non-identity pseudo-reflection  then $[V, g] \subseteq V_i$ for some $i$. Let $W_i$ be the subgroup of $W$ generated by the pseudo-reflections $g$ such that $[V, g] \subset V_i$ (if there is no such pseudo-reflections then $W_i=Id$). The subgroup $W_i$ acts trivially on all $V_j$, $j \neq i$, and by By Lemma \ref{lem-boring}, (2),
$W_i$ and $W_j$ commute. Therefore, $W$ is the direct product of the subgroups $W_i$'s, and each $W_i$ is irreducible pseudo-reflection group on $V_i$, or trivial.
\end{proof}

Consider now the Weyl algebra $A_n(\k)$ with a linear action of   a  pseudo-reflection group $W$ extended from a linear action on $n$-dimensional vector space $V$. 
By Proposition \ref{coxeter-like} we have $W \simeq W_1 \times \ldots \times W_m$. 
Suppose that $n=n_1+ \ldots + n_m + k$ and $A_n(\k)=A_{n_1}(\k) \otimes  \ldots \otimes A_{n_m}(\k) \otimes A_k(\k)$. Then for each $i=1, \ldots, m$, $W_i$ acts on $A_{n_i}(\k)$ and fixes all 
$A_{n_j}(\k)$ with $j\neq i$. The whole group $W$ fixes $A_k(\k)$.  Then we have 

$$A_n(\k)^W \simeq A_{n_1}(\k)^{W_1} \otimes \ldots \otimes  A_{n_m}(\k)^{W_m} \otimes A_k (\k).$$

Applying Corollaey \ref{NNP-irred} we immediately obtain

\begin{theorem} \label{NNP-pseudo-reflections}
The NNP holds for all pseudo-reflection groups over fields of zero characteristic.
\end{theorem}

\begin{remark}


We would like to point that
the solution of the NNP for irreducible pseudo-reflection groups in  Proposition \ref{rel-Weyl} uses an explicit computation  of Weyl generators which realized the isomorphism $F_n(\k)^W \cong F_n(\k)$. Since an algebraic basis of the invariants of the irreducible unitary reflection groups is well known (\cite{Kane}); we can compute these Weyl generators, refining the results in \cite{Eshmatov}.
We consider an example below. 
\end{remark}

\begin{example}
Assume $n=3$ and $W=S_n$.
Set $J=(x_1-x_2)(x_2-x_3)(x_3-x_2)$. The following elements are  the Weyl generators of $F_3(\k)^{S_3}$, where $S_3$ acts by permutations:

\[ x_1 + x_2 + x_3 \rightarrow X_1, \, x_1 x_2 + x_2 x_3 + x_1 x_3 \rightarrow X_2, \, x_1 x_2 x_3 \rightarrow X_3; \]

\[ \frac{x_1^2(x_2-x_3)}{J} \partial_1 + \frac{x_2^2(x_3-x_1)}{J} \partial_2 + \frac{x_3^2(x_1-x_2)}{J} \partial_3 \rightarrow Y_1; \]
\[ \frac{x_1(x_3-x_2)}{J} \partial_1 + \frac{x_2(x_1 - x_3)}{J} \partial_2 + \frac{x_3(x_2-x_1)}{J} \partial_3 \rightarrow Y_2;\]
\[ \frac{(x_2-x_3)}{J} \partial_1 + \frac{(x_3-x_1)}{J} \partial_2 + \frac{ (x_1-x_2)}{J} \partial_3 \rightarrow Y_3. \]

Here, $Y_i X_j-X_j Y_i= \delta_{ij}$ for $i, j=1, 2, 3$.
\end{example}

\section{Invariant cross products}

In this section we apply the result above to the subalgebras of invariants of cross products.

Let  $G$ be afinite group of automorphisms of field $L$, 
$\mathcal M$  a monoid  of automorphisms of  $L$ on which $G$
acts by conjugations.  Denote by $L*\mathcal M$ the cross product, where $(l m)(l' m')=(lm(l'))(m m')$ for $l, l'\in L$ and $m, m'\in \mathcal M$.
 We have a well defined action of $G$ on the cross product $L*\mathcal M$: $g(l m)=g(l)g(m)$, $g \in G$, $l \in L$, $m \in \mathcal{M}$. 
 Consider the ring of invariants $(L*\mathcal{M})^G$ by the action of $G$.

Suppose $L*\mathcal M$ is an Ore domain. Then $(L*\mathcal M)^G$ is an Ore
domain and the skew field of fractions $\mathcal F ((L*\mathcal M)^G)$ is isomorphic to $(\mathcal F (L*\mathcal M))^G$ with induced actions of $G$ on the skew field of fractions.

Assume $L\simeq \k (t_{1},\dots, t_{n})$ to be the field of fractions of the symmetric algebra $S(V)$ for some $n$-dimensional $\k$-vector space $V$. If $G<GL_n$ is a finite group then it acts linearly on $L$. If  $G$ normalizes  $\mathcal M$ then $(L*\mathcal M)^G$ is a  \emph{linear Galois order} \cite{Eshmatov}.  Suppose $L=\k (t_1, \ldots, t_n;  z_1, \ldots z_m)$, for some integers $n$, $m$, $\mathcal M\simeq \mathbb Z^{n}$ 
 is generated by $\varepsilon_1, \ldots \varepsilon_1$, where
 $\varepsilon_i(t_j)=t_j + \delta_{ij}$, $\varepsilon_i(z_k)=z_k$  $i,j=1, \ldots, n$, $k=1, \ldots, m$ (in this case we ay that $\mathcal M$ acts bu \emph{shifts} on $L$).
 It was shown in \cite{Eshmatov}, Theorem 6, that for such cross products and for any complex reflection group $G$,  $(L*\mathcal M)^G$  is birationally equivalent to $A_n(\mathbb C)\otimes \mathbb C [z_1, \ldots, z_m]$.  Applyng Theorem \ref{thm-main} we can extend this result to other groups.

\begin{theorem}
\label{thm-cross}
Let $\k$ has characteristic zero, 
 $L=\k (t_1, \ldots, t_n;  z_1, \ldots z_m)$, for some integers $n$, $m$, and    $\mathcal M\simeq \mathbb Z^{n}$ acting by shifts on $L$. Then
\begin{itemize}
\item[{(1)}] $(L*\mathcal M)^G$
 is birationally equivalent to $A_n(\k)\otimes \k [z_1, \ldots, z_m]$ for any pseudo-reflection group $G$;
 \item[{(2)}] If  $L^G\simeq L$ for a given group $G$ then   $(L*\mathcal M)^G$     is birationally equivalent to $A_n(\k)\otimes \k [z_1, \ldots, z_m]$.
 \end{itemize}
\end{theorem}

\begin{proof}
Indeed,we have  an embedding of the Weyl algebra  $A_n(\k)$ to $  \k[t_1, \ldots, t_n]*\mathbb Z^n$ and their skew fields of fractions are isomorphic. 
Hence, item (1) follows from Theorem \ref{NNP-pseudo-reflections}.  If $L^G\simeq L$ then the CNP holds and (2) follows from Theorem \ref{thm-main}.

\end{proof}

We finish with the following problems:

\

\noindent{ \bf Problem}: Find an example of a linear action of a finite group such that the CNP does not hold but the NNP holds.
Find an example of a linear action of a finite group for which the NNP fails. 

\

\end{document}